\documentclass[11pt,reqno]{amsart}

\usepackage{amsmath,amssymb,amsthm,amscd}
\usepackage{calrsfs,enumitem}
\usepackage[dvipsnames]{xcolor}
\usepackage{mathtools}
\usepackage{graphicx}

\usepackage[linktocpage]{hyperref}
\hypersetup{
    colorlinks=true,
    linkcolor=blue,
    hypertexnames=false,
    filecolor=black,      
    urlcolor=blue,
    citecolor=red
}

\numberwithin{equation}{section}

\usepackage[]{cleveref}
\theoremstyle{plain}
\newtheorem{theorem}{Theorem}[section]

\newtheorem{lemma}[theorem]{Lemma}
\newtheorem{corollary}[theorem]{Corollary}

\newtheorem{conjecture}[theorem]{Conjecture}
\newtheorem{question}[theorem]{Question}
\newtheorem{claim}{Claim}

\theoremstyle{definition}

\theoremstyle{remark}
\newtheorem{remark}[theorem]{Remark}

\def\N{\mathbb{N}}
\def\R{\mathbb{R}}

\def\X{X_{1}}
\def\Y{X_{2}}

\DeclareMathOperator{\area}{area}
\DeclareMathOperator{\Isom}{Isom}

\def\G{\Gamma}

\def\g{\gamma}

\def\geo{\partial}

\DeclareMathOperator{\slp}{sl}

\title{Rigidity of convex co-compact diagonal actions}

\author{Subhadip Dey \and Beibei Liu}

\begin{document}
\maketitle
\begin{abstract}
Kleiner-Leeb and Quint showed that convex subsets in higher-rank symmetric spaces are very rigid compared to rank 1 symmetric spaces. Motivated by this, we consider convex subsets in products of proper CAT(0) spaces $X_1\times X_2$ and show that for any two convex co-compact actions $\rho_i(\G)$ on $X_i$, where $i=1, 2$,  if the diagonal action of $\G$ on $X_1\times X_2$ via $\rho=(\rho_1, \rho_2)$ is also convex co-compact, then under a suitable condition, $\rho_1(\G)$ and $\rho_2(\G)$ have the same marked length spectrum.
 
\end{abstract}

\section{Introduction}

Let $X$ be a proper CAT(0) space, and $\G$ be an abstract group.
A properly discontinuous action of $\G$ on $X$ by isometries is called {\em convex co-compact} if there exists a nonempty closed $\G$-invariant convex subset $C\subset X$ such that the quotient $C/\G$ is compact.
The action of $\G$ on this convex set carries many essential algebraic, geometric, and dynamical information about $\G$. A smallest such quotient $C/\G$ is called a {convex core} of $X/\G$. 
Convex co-compact (and, more generally, geometrically finite) isometry groups of rank 1 symmetric spaces have been extensively studied, exhibiting rich geometric and dynamical characteristics.
While there are many interesting examples of convex co-compact actions on rank 1 symmetric spaces and, more broadly, on Hadamard manifolds with negatively pinched curvatures \cite{bow3}, Kleiner-Leeb \cite{KLeeb} and Quint \cite{Quint} showed that for higher rank symmetric spaces $X$, discrete group invariant convex subsets of $X$ are highly constrained. 

This article is motivated by the following question of Delzant \cite[Problem 22]{problemlist}:

\begin{question}\label{ques:delzant}
Suppose that $\X$ and $\Y$ are two $\textup{CAT}(0)$ spaces for $i=1, 2$, and there exist representations $\rho_i: \G\rightarrow \Isom(X_i)$, $i=1,2$, such that $\rho_i(\G)$ acts on $X_i$ properly discontinuously and co-compactly. 
Does there exist a compact convex core for the diagonal action of $\G$ on $X_1\times X_2$?
\end{question}

\Cref{ques:delzant} has a negative answer in general: For example, if $S$ is a closed orientable surface of genus at least two and $\rho_1,\rho_2 : \pi_1(S)\to \operatorname{PSL}(2,\R)$ are two distinct (up to conjugation by elements in $\operatorname{PSL}(2, \R)$) discrete and faithful representations, then $\rho_1$ and $\rho_2$ have distinct marked length spectra; see Thurston \cite[\S 3]{thurston1998minimal}.
Although $\rho_1$, $\rho_2$ are co-compact, it follows from \Cref{thm:main} below that the induced diagonal action of $\pi_1(S)$  on $\mathbb{H}^2\times\mathbb{H}^2$ cannot be convex co-compact.
Nevertheless, it would be interesting to find exact conditions under which \Cref{ques:delzant} has a positive answer. For example, see \Cref{cor} below.

Our main result is as follows:

\begin{theorem}\label{thm:main}
    Let $X_1$ and $X_2$ be proper $\textup{CAT}(0)$ spaces, $\G$ be an abstract group acting properly discontinuously on $X_1$ and $X_2$ by isometries, and let $\rho_i : \G \to \Isom(X_i)$ for $i=1,2$ denote the induced homomorphisms. Suppose that
    there exists an element $\alpha\in \G$ such that $\rho_i(\alpha)$ is a rank one isometry of $X_i$, $i=1,2$. If the action of $\G$ on $X_1\times X_2$ via $(\rho_1,\rho_2)$ is convex co-compact, then, up to scaling, $\rho_1$ and $\rho_2$ have the same marked length spectrum.
\end{theorem}

\begin{remark}
If $C\subset X_1\times X_2$ is a $\rho(\G)$-invariant convex set such that $C/\rho(\G)$ is compact, then $\rho_i(\G)$ acts co-compactly on the projection of $C$ to $X_i$ for $i=1,2$. Thus,  the convex co-compactness of the diagonal action implies that both $\rho_1(\G)$, $\rho_2(\G)$ are convex co-compact in $X_1, X_2$ respectively. 
\end{remark}

In the setting of Theorem \ref{thm:main}, for every infinite order element $\g\in \G$, $\rho_1(\g), \rho_2(\g)$ and $\rho(\g)=(\rho_1(\g), \rho_2(\g))$ are hyperbolic isometries. We associcate a vector $(\ell(\rho_1(\g)), \ell(\rho_2(\g))) \in \R^2$ to each element $\g\in \G$ where $\ell(\rho_i(\g))$ is the translation length of $\rho_i(\g)$ in $X_i$ for $i=1,2$. That the representations $\rho_1$ and $\rho_2$ have the same marked length spectrum up to scaling means that there exists $\lambda>0$ such that for all $\g\in \G$, $\ell(\rho_1(\g))=\lambda \,\ell(\rho_2(\g))$. For the related definitions, see Section \ref{sec:pre}. 

Notably, the well-known marked length spectrum conjecture of Burns and Katok \cite[Problem 3.1]{BKatok} states that if two closed manifolds $(M, g)$ and $(M_0, g_0)$ with negative sectional curvatures have the same marked length spectrum, then they are isometric. 
Therefore, a positive answer to this conjecture will imply the following:

\begin{conjecture}\label{conj}
    Let $(X_1,d_1)$ and $(X_2,d_2)$ be negatively curved Hadamard manifolds and let $\G$ be a torsion-free group acting properly discontinuously and co-compactly on $X_1$ and $X_2$.
    If the diagonal action of $\G$ is convex co-compact, then there exists a $\G$-equivariant isometry $\phi: (X_1,d_1) \to (X_2,\lambda d_2)$ for some $\lambda>0$.
\end{conjecture}

The marked length spectrum conjecture has been settled in dimension 2 by Otal \cite{Otal} and, independently, Croke \cite{Croke} (see also \cite{CFF}). 
In higher dimensions,  Hamenst\"adt \cite{Hamen} proved this conjecture when $g_0$  is a locally symmetric metric. A recent breakthrough of Guillarmou and Lefeuvre proved a certain local version of the marked length spectrum conjecture \cite{GF}.  Combining Theorem \ref{thm:main} with some of the known cases of the marked length spectrum conjecture, one obtains:

\begin{corollary}\label{cor}
Let $(X_1,d_1)$ and $(X_2,d_2)$ be negatively curved Hadamard manifolds, $\G$ be a torsion-free group acting properly discontinuously and co-compactly on $X_1$ and $X_2$, and $\rho_i:\G \to \operatorname{Isom} (X_i)$, $i=1,2$, be the induced homomorphisms.
Suppose that one of the following holds:
\begin{enumerate}
    \item $\dim X_1 = 2$.
    \item $X_1$ is a symmetric space. 
\end{enumerate}
Then, the diagonal action $\rho=(\rho_1, \rho_2)$ on $X_1\times X_2$ is convex co-compact if and only if there exists a $\G$-equivariant isometry $\phi: (X_1,d_1) \to (X_2,\lambda d_2)$ for some $\lambda>0$. 

\end{corollary}

\begin{remark}
 The marked length spectrum conjecture has been formulated and studied for the class of locally compact CAT(-1) spaces, see for example \cite{HP97,BL17,AB87,CM87,CL,Lafont1}. Interested readers could combine our main theorem \ref{thm:main} to various settings where the marked length spectrum conjecture is formulated.  When $X_i$ are CAT(-1) spaces where $i=1,2$, the assumption of rank 1 isometries in Theorem \ref{thm:main} is vacuously satisfied. 
\end{remark}

\subsection*{Acknowledgement} We would like to thank Hee Oh and  Shi Wang for interesting discussions. We are also grateful to Jean-Franc\c{o}is Lafont for helpful comments on our earlier draft.  The second author is partially supported by the NSF grant DMS-2203237.

\section{Isometries of CAT(0) spaces}\label{sec:pre}

Let $X$ be a proper CAT(0) space. The \emph{displacement function} $d_g: X\rightarrow \mathbb{R}_{+}$ of an isometry $g$ of $X$ is defined by $d_g(x)\coloneqq 
d(x, gx)$. The \emph{translation length} of $g$ is the number $\ell(g)\coloneqq 
\inf\{ d_{g}(x) \mid x\in X\} $. Based on whether $d_{g}$ attains the infimum $\ell(g)$, the isometries of $X$ are classified as follows (see for example \cite[Definition 6.3]{Bridson-Haefliger}). 
An isometry $g : X\to X$ is called:  
\begin{enumerate}
    \item {\em parabolic}  if $ d_{g}(x)$ does not attain its infimum $\ell(g)$ for  any point  $x\in X$.
    \item {\em semi-simple} if it is not parabolic. Semi-simple isometries are further classified as follows. A semisimple isometry $g$ of $X$ is called:
    \begin{enumerate}
    \item {\em elliptic} if it has a fixed point in $X$.
    \item {\em hyperbolic} if it has no fixed points in $X$. In this case, there is a (possibly non-unique) {\em axis} of $g$, i.e., a $g$-invariant complete geodesic line $l_g$ in $X$ where the displacement function $d_g$ achieves its minimum. Moreover, any $g$-invariant complete geodesic in $X$ is an axis of $g$ and any two axes of $g$ are parallel.
    \end{enumerate}
\end{enumerate}

Furthermore, a hyperbolic isometry $g$ is called {\em rank one} if an (equivalently, any) axis $l_g$ of $g$ has the {\em contraction property}, i.e., there exists a constant $K>0$ such that the nearest-point projection into $l_g$ of any metric ball in $X$, disjoint from $l_g$,  has diameter at most $K$.
    Equivalently (by \cite[Thm. 5.4]{bestvina2009characterization}), $l_g$ does not bound a flat half-plane in $X$.
Rank one isometries $g: X\to X$ act with north-south dynamics on the visual boundary $\geo X$ \cite[Lem. 4.4]{hamenstadt2009rank}. More precisely, there are two fixed points $g_\pm \in \geo X$ such that
\[
 g^n\vert_{\geo X\setminus \{g_-\}} \to \operatorname{const}_{g^+}
 \quad \text{as } n\to \infty
\]
uniformly on compact sets. Here $\geo X$ denotes the visual boundary of $X$ which, as a set, is the equivalence class of geodesic rays. The topology on $\geo X$ is the {\em cone topology}; see \cite{Bridson-Haefliger} for the definition of this topology and further details on it.


\begin{lemma}\label{lem:hyper}
    If $\G<\Isom(X)$ is convex co-compact, then any infinite order element $\g$ of $\G$ is a hyperbolic isometry of $X$. 
    
    Moreover, if $C\subset X$ is a nonempty $\G$-invariant convex set such that $C/\G$ is compact, then $C$ contains an axis of $\g$.
\end{lemma}

\begin{proof}
    Since $C/\G$ is compact and $C$ is CAT(0) (being a convex subset of a CAT(0) space), the restriction $\g\vert_C$ of $\g$ to $C$ is semisimple (see \cite[Prop. II.6.10]{Bridson-Haefliger}).
    Moreover, since $\G$ acts on $X$ and, in particular, on $C$ properly discontinuously, $\g\vert_C$ is not elliptic.
    So, $\g\vert_C$ is hyperbolic.
    Finally, since the nearest point projection map $X\to C$ is a $\G$-equivariant $1$-Lipschitz map, we have $\inf\{ d(x,\g x):\ x\in X \} = \inf\{ d(x,\g x):\ x\in C \}$.
    So, it follows that any axis of $\g\vert_C$ in $C$ is also an axis of $\g$ in $X$.
    Thus $\g$ is a hyperbolic isometry of $X$ and it has an axis in $C$.
\end{proof}


\subsection{Slope}
Given two proper CAT(0) metric spaces $(\X, d_1)$ and $(\Y, d_2)$, let $r: [0,\infty) \to X_1\times X_2$ be a geodesic ray.
The projection of this ray into a factor $X_i$ is either a point or a geodesic ray (after a reparametrization), see \cite[Prop. I.5.3]{Bridson-Haefliger}. 
We define the {\em slope of $r$}, $\slp(r)\in [0,\infty]$, as follows:
\begin{enumerate}
    \item If the projection of (the image of) $r$ into $X_2$ is a point, then $\slp(r) \coloneqq \infty$.
    \item Otherwise, 
    \begin{equation}\label{eqn:slp_ray}
        \slp(r) \coloneqq \frac{d_{1}(r_1(0),\, r_1(1) )}{d_{2}(r_2(0),\, r_2(1) )},
    \end{equation}
    where $r_1$ (resp. $r_2$) denotes the composition of $r$ with the projection of $X_1\times X_2$ into $X_1$ (resp. $X_2$).
\end{enumerate}

Similarly, if $l: (-\infty,\infty) \to X_1\times X_2$  is a geodesic line, then the slope of $l$ is defined as:
\[
 \slp(l) : = \slp(l\vert_{[0,\infty)}).
\]

\begin{remark}
    In \eqref{eqn:slp_ray}, one can replace $r_1(1)$ and $r_2(1)$ by $r_1(t)$ and $r_2(t)$, respectively, for any positive number $t$. Note that the value of the ratio will remain unchanged since the projections of $r$ to the factors $X_1$ and $X_2$ have constant speeds.
    Moreover, if $l$ is a geodesic line, then $\slp(l\vert_{(-\infty,0]}) = \slp(l) = \slp(l\vert_{[0,\infty)})$; so, $\slp(l)$ could also be defined as $\slp(l\vert_{(-\infty,0]})$.
\end{remark}

\begin{lemma}\label{lemma:slp}
    Asymptotic rays in $X_1\times X_2$ have the same slope.
\end{lemma}

\begin{proof}
Let $r, r' : [0,\infty) \to \X\times \Y$ be unit speed parameterized geodesic rays that are asymptotic, that is, there exists a constant $K \ge 0$ such that $d(r(t), r'(t))\leq K$ for all $t\geq 0$. 
We let $r_i, r'_i$ denote the projections of $r, r'$ in $(X_i, d_i)$ respectively for $i=1,2$. We claim that the Hausdorff distance between $r_i$ and $r'_i$ is bounded from above for both $i=1,2$ and $t\geq 0$. This is straightforward by observing that 
\begin{equation}\label{eqn1:lemma:slp}
d^2_{{1}}(r_1(t), r'_1(t))+d^2_{{2}}(r_2(t), r'_2(t)) = d^2(r(t), r'(t)) \leq K^2.
\end{equation}
If the projection of $r$ to $X_2$ (resp. $X_1$) is a point, then the projection of $r'$ to $X_2$ (resp. $X_1$) is also a point, and vice versa. In this case,
$$\slp(r)=\slp(r')=\infty \quad (\text{resp. } 0).$$
Otherwise, the projections of $r, r'$ to $X_1$ and $X_2$ are not points. In that case, $\slp(r),\slp(r')\in (0,\infty)$. Since 
$$d_{1}(r_1(0), r_1(t))=\slp(r) d_{2}(r_2(0), r_2(t)), \quad d_{1}(r'_1(0), r'_1(t))=\slp(r') d_{2}(r'_2(0), r'_2(t)).$$
for all $t\in\R$, we have that \[ \frac{\slp(r)}{\slp(r')} = \frac{d_{1}(r_1(0), r_1(t))}{d_{1}(r'_1(0), r'_1(t))}\cdot \frac{d_{2}(r'_2(0), r'_2(t))}{d_{2}(r_2(0), r_2(t))}  \]
By \eqref{eqn1:lemma:slp}, each ratio in the right-hand side of the above equation converges to $1$ as $t\to\infty$. Hence,  ${\slp(r)}={\slp(r')}$.
\end{proof}

For a hyperbolic isometry $g$ of the product of CAT(0)-spaces $X_1\times X_2$ , define the {\em slope} of $g$ by
\begin{equation}
    \slp (g) \coloneqq \slp (l_{g})
\end{equation}
where $l_g$ is an axis of $g$. The definition is independent of the choices of axes: Since any two axes of $g$ are parallel, by \Cref{lemma:slp}, they must have the same slope.






\section{Proof of \Cref{thm:main}} 
\label{sec:main}

Throughout this section, let $(X_i, d_i)$, $i=1,2$, be proper  CAT(0) spaces. Let $\G$ be an infinite abstract group acting on $X_i$, $i=1,2$, convex co-compactly by isometries and let $\rho_i: \G\rightarrow \Isom(X_i)$ denote the induced homomorphisms.
Thus, we have a diagonal action of $\G$ by isometries on the product $(X_1\times X_2, d_1\times d_2)$ given by the homomorphism $\rho : \G \to \Isom (X_1\times X_2)$, $\rho(\gamma) = (\rho_1(\gamma),\rho_2(\gamma))$, $\g\in\G$.
We assume that this action is convex co-compact, i.e.,
there exists a nonempty $\rho(\G)$-invariant convex subset $C\subset \X\times \Y$ such that the quotient $C/\rho(\G)$ is compact. 

The goal of this section is to prove \Cref{thm:main}. That is, under the following additional assumption, $\rho_1$ and $\rho_2$ have the same marked length spectrum up to a scalar multiple:
there exists an element $\alpha \in \Gamma$ such that $\rho_1(\alpha)$ and $\rho_2(\alpha)$ are rank-one isometries of $X_1$ and $X_2$, respectively.

Fix a base-point $x=(x_1, x_2)\in C$.
For $i=1,2$, let $l_i$ denote a complete geodesic line in $X_i$ passing through $x_i$.
Then 
\begin{equation}\label{eqn:flatF}
    F = l_1\times l_2\subset X_1\times X_2
\end{equation} is a $2$-flat (i.e., isometric to $\R^2$ with the Euclidean metric) passing through $x$.
The intersection
\[
C_F = F \cap  C
\]
is convex in $X_1\times X_2$.

\begin{lemma}\label{lem:thm:main}
Suppose that $C_F$ has an infinite diameter. Then, $C_F$ lies within a finite Hausdorff distance away from a geodesic ray or a complete geodesic line in $F$.
\end{lemma}

\begin{proof}
 For $D>0$, let $f_D: [0,\infty) \to [0,\infty)$,
 \[
  f_D(r) = \frac{\#\{\g\in\G :\ d(\rho(\g) x,F) \le D,\ d(x,\rho(\g) x)\le r \}}{r}.
 \]
 
 \begin{claim}\label{claim:one} We have that
 \begin{equation*}
     \limsup_{r\to\infty} \frac{f_D(r)}{r} <\infty.
 \end{equation*}
 \end{claim}

 \begin{proof}[Proof of claim]
     Let $\bar f_D: [0,\infty) \to [0,\infty)$,
 \[
  \bar f_D(r) = \frac{\#\{\g\in\G :\ d_{1}(\rho_{1}(\g) x_1,l_1) \le D,\ d_{1}(x_1,\rho_1(\g) x_1)\le r \}}{r},
 \]
 where $x_1$ and $l_1$ are as above (see \eqref{eqn:flatF}).
 
 Note that $d(\rho(\g) x,F) \le D$ implies $d_{1}(\rho_{1}(\g) x_1,l_1) \le D$, and
 similarly, $d(x,\rho(\g) x)\le r$  implies  $d_{1}(x_1,\rho_1(\g) x_1)\le r$.
 Consequently, we have that $f_D\le \bar f_D$. Thus, in order to show that $\limsup_{r\to\infty} {f_D(r)}/{r} <\infty$, it will be sufficient to show that $\limsup_{r\to\infty} \bar f_D(r)/r <\infty$. 
 
 Split $l_1$ as a union of geodesic rays $l^{+}_1$, $l_1^-$ emanating from $x$.
 Then, for all $n\in \N$, the set
 \[
  R^\pm_n = \{ y\in X_1 :\ d_{1}(y,l^\pm_1) \le D, 
  \ n\le d_{1}(y,x_1)\le n+1\} \quad \subset X_1
 \]
 has a diameter bounded above by $2D + 1$.
 Since the $\rho_1(\G)$-orbit of the point $x_1\subset X_1$ is uniformly separated in $X_1$, there exists $N>0$ such that for all $n\in\N$, the sets $R^\pm_n$ contain at most $N$ points from $\rho_1(\G) x_1$.
Therefore, $\bar f_D(n) \le  nN$ for all $n\in \N$ and hence $\limsup_{r\to\infty} \bar f_D(r)/r \le N$, which proves the claim.
 \end{proof}

\begin{claim}\label{claim:two} We have that
\begin{equation*}
 \limsup_{r\to\infty}\frac{\mathrm{area}(C_F \cap B_r(x))}{r} <\infty,
\end{equation*}
where $\mathrm{area}$ denotes the standard Euclidean area in $F \cong \R^2$.
\end{claim}

\begin{proof}[Proof of claim]
    Since $C/\G$ is compact, the orbit $\rho(\G)x$ is $D$-dense in $C$, that is, for any point $y\in C$, $d(y,\rho(\G)x)\leq D$ for some constant $D >0$.
 Hence, $C_F \cap B_r(x)$ is covered by $D$-balls centered at the orbit points in $\rho(\G) x$, which lie in the $D$-neighborhood of $C_F \cap B_r(x)$. 
 The intersection of each ball with $F$ is contained in a Euclidean disk of radius $D$ in $F$, and thus has a maximum area of $2\pi D^2$.
 By \Cref{claim:one}, we have that
 $$ \limsup_{r\to\infty}\frac{\mathrm{area}(C_F \cap B_r(x))}{r}\leq \limsup_{r\to\infty}2\pi D^2 \dfrac{f_D(r)}{r}<\infty.$$
 This proves the claim.
\end{proof}

Now we return the proof of the lemma.
Since $C_F$ is a convex subset of $X_1\times X_2$ of infinite diameter, it contains a geodesic ray $l$ emanating from the basepoint $x$.

We split the proof into two cases:

\medskip\noindent
{\bf Case 1.}
Suppose for now that $C_F$ does not contain the bi-infinite geodesic $\bar l$ extending $l$. In this case, we show that $C_F$ is within a finite Hausdorff distance from $l$. We will argue by contradiction:

Suppose that the Hausdorff distance between $C_F$ and $l$ is infinite. 
Then, there exists a sequence of points $w_n\in  C_F$ such that $d(w_n, l)\rightarrow \infty$, as $n\to\infty$.
If $\liminf_{n\to\infty} d(w_n,\bar l)$ were finite, then due to the convexity of $C_F$, we would have $\bar l\subset C_F$, which would contradict our initial assumption that $C_F$ does not contain the bi-infinite geodesic $\bar l$ extending $l$. Thus,
$d(w_n,\bar l) \to \infty$ as $n\to \infty$.  Let $v_n\in \bar l$ be the nearest point projection of $w_n$ to $\bar l$. So, $d(w_n, v_n)\rightarrow \infty$ as $n\to\infty$. Let $r_n:= d(x, w_n)$ and  $l_n=B_{r_n+1}(x)\cap l$.  By the convexity of $C_F$, the triangle $\triangle_n$ with one side $l_n$ and vertex $w_n$ is contained in $C_F$. We have that
\[
\dfrac{\area(C_F\cap B_{r_{n}+1}(x))}{r_n+1}
\geq \dfrac{\area\triangle_n}{r_{n}+1}
=\dfrac{ d(w_n, \bar l) \operatorname{length}(l_n)}{2 (r_n+1)}=\dfrac{d(w_n, \bar l)}{2}.
\]
Since the quantity $\dfrac{d(w_n, \bar l)}{2} \to \infty$ as $n\rightarrow \infty$,
we have a contradiction with \Cref{claim:two}. 

\medskip\noindent
{\bf Case 2.}
Suppose that $C_F$ contains the bi-infinite geodesic extension $\bar l$ of $l$.
If, to the contrary, $C_F$ is not within a finite Hausdorff distance away from $\bar l$,
then there exists a sequence $w_n\in C_F$ such that $d(w_n,\bar l)\to \infty$ as $n\to\infty$.
Due to the convexity of $C_F$, it follows that $C_F$ encloses the convex hull $C_n$ of $\bar l$ and $w_n$, which is a bi-infinite rectangular strip with $\bar l$ as one of the sides.
Since the width of $C_n$, $d(w_n,\bar l)$, goes to infinity, it follows that $C_F$ contains a half-plane.
Thus, the area of $C_F\cap B_r(x)$ grows quadratically with $r$, contradicting \Cref{claim:two} again.

\medskip

The proof of the lemma is now complete.
\end{proof}

As  $\rho_i(\G)$ acts co-compactly on $X_i$, each infinite order element $\g\in \G$ acts on $X_i$ as a hyperbolic isometry  (see Lemma \ref{lem:hyper}).
That is, there exists a $\rho_i(\g)$-invariant complete geodesic axis $l_{\rho_i(\g)} \subset  X_i$, 
such that $\rho_i(\g)$ acts by a translation along $l_{\rho_i(\g)}$. 

\begin{lemma}\label{lem:flat}
    If $\rho(\g)$ is a hyperbolic isometry of $X_1\times X_2$ and $l_{\rho(\g)}\subset X_1\times X_2$ is an axis of $\rho(\g)$, then there exists a unique axis $l_{\rho_i(\g)} \subset X_i$ for $\rho_i(\g)$ with $i=1,2$, such that 
    \begin{equation}\label{eqn:axis_in_flat}
 l_{\rho(\g)} \subset F_{\rho(\g)} \coloneqq l_{\rho_1(\g)} \times l_{\rho_2(\g)}.
\end{equation}

Moreover, $\slp(\rho(\g))=\dfrac{\ell(\rho_1(\g))}{\ell(\rho_2(\g))}$. 
\end{lemma}

\begin{proof}
Let $l_i$ be the projection of $l_{\rho(\g)}$ to $X_i$ where $i=1,2$. Note that $l_i$ is not a point as $l_{\rho(\g)}$ is $\rho(\g)$-invariant. So both $l_1$ and  $l_2$ are complete geodesic lines in $X_1, X_2$ respectively. Moreover, $l_i$ is $\rho_i(\g)$-invariant. So, $l_i$ is an axis for $\rho_i(\g)$ for $i=1,2$. 
That is, for any point $x_i\in l_i$, $d_i(x_i, \rho_i(\g) x_i)=\ell(\rho_i(\g))$. 


Given $(x, y)\in l_{\rho(\g)}$, then $\slp(\rho(\g))=\dfrac{d_1(x, \rho_1(\g) x)}{ d_2(y, \rho_2(\g) y)}=\dfrac{\ell(\rho_1(\g))}{\ell(\rho_2(\g))}.$
\end{proof}


Let $\alpha\in\G$ be an infinite order element such that $\rho_i(\alpha)$ is a rank one isometry of $X_i$ for $i=1,2$.
\Cref{thm:main} can be rephrased as follows:

\begin{theorem}\label{thm:main_rephrased}
    For all infinite order elements $\beta\in\G$,
    $\slp(\rho(\alpha)) = \slp(\rho(\beta))$.
\end{theorem}

Before discussing the proof, we prove a special case as follows:

\begin{lemma}\label{lem:thm:main_rephrased}
    Let $\beta\in\G$ be an infinite order element.
    Suppose that $\rho_i(\beta)$ fixes the attractive and repulsive fixed points of $\rho_i(\alpha)$ in $\geo X_i$ for $i=1,2$.
    Then $\slp(\rho(\beta)) = \slp(\rho(\alpha))$.
\end{lemma}

\begin{proof}
Let $i=1,2$.
Pick an axis $l_{\rho(\alpha)}$ 
in $X_1\times X_2$ and let $l_i$ be the axis of $\rho_i(\alpha)$ obtained by projecting $l_\rho(\alpha)$ into $X_i$.
The parallel set\footnote{The {\em parallel set} $P(l)$ of a given complete geodesic $l$ in a metric space $Y$ is the union of all complete geodesics in $Y$ parallel (i.e., bi-asymptotic) to $l$.} $P(l_i)$ of $l_i$
metrically decomposes as $\operatorname{CS}(l_i) \times \R$, where $\operatorname{CS}(l_i)$ is the cross-section of $P(l_i)$ (see \cite[Thm. II.2.14]{Bridson-Haefliger}). The cross-section $\operatorname{CS}(l_i)$ can be realized as a convex subset of $X_i$ as the set of all nearest point projections of a fixed $x\in l_i$ to the lines parallel to $l_i$.
Since $\rho_i(\alpha)$ is rank one, $l_i$ does not bound a flat half plane and hence $\operatorname{CS}(l_i)$ has a finite diameter.

Since $\rho_i(\beta)$ fixes the ideal endpoints of $l_i$, $\rho_i(\beta)$ preserves the family of lines parallel to $l_i$.
Hence, $P(l_i)$ is $\rho_i(\beta)$-invariant.
The action $\rho_i(\beta)$ on $P(l_i)$ induces an isometric action on $\operatorname{CS}(l_i)$ via the product decomposition $P(l_i) = \operatorname{CS}(l_i) \times \R$.
Since $\operatorname{CS}(l_i)$ is a CAT(0) metric space with a finite diameter, $\rho_i(\beta)$ fixes the center of $\operatorname{CS}(l_i)$.
This center corresponds to a line $l'_i$ parallel to $l_i$, which is then an axis of $\rho_i(\beta)$.
(In particular, $\rho_i(\beta)$ is rank one since $l'_i$ has contracting property.)
So, $\rho_i(\beta)$ acts on the $\R$ factor of $P(l_i) = \operatorname{CS}(l_i) \times \R$ by translation by an amount of $\pm\ell(\rho_i(\beta))$.

Now we finish the proof of the lemma by a contradiction.
Suppose that $\slp(\rho(\beta)) \ne \slp(\rho(\alpha))$.
Consider the properly discontinuous action of $\langle\alpha,\beta\rangle < \Gamma$ on $P(l_1)\times P(l_2)$ via $\rho$.
The projection $P(l_1)\times P(l_2) \to \R\times \R$, induced by the metric decomposition $P(l_i) = \operatorname{CS}(l_i) \times \R$, gives rise to an action of $\langle\alpha,\beta\rangle$ by translations on $\R\times \R$; more precisely, $\rho(\alpha) (a,b) = (a,b) + (\ell(\rho_1(\alpha)),\ell(\rho_2(\alpha)))$ and, similarly, $\rho(\beta) (a,b) = (a,b) + (\pm\ell(\rho_1(\beta)),\pm\ell(\rho_2(\beta)))$.
In particular, since $\slp(\rho(\beta)) \ne \slp(\rho(\alpha))$,
the image of $\langle\alpha,\beta\rangle$ in $\operatorname{Isom}(\R^2)$ is a lattice, i.e., $\langle\alpha,\beta\rangle$ acts geometrically on $\R^2$.
Hence, by the Schwarz-Milnor Lemma, $\langle\alpha,\beta\rangle$ is quasi-isometric to $\R^2$.

On the other hand, $\langle\alpha,\beta\rangle$ acts geometrically on $P(l_1)$ via $\rho_1$ and, since $P(l_1)$ is quasi-isometric to $\R$ (as $\operatorname{CS}(l_1)$ has a finite diameter), 
$\langle\alpha,\beta\rangle$ is quasi-isometric to $\R$.
This is a contradiction to the conclusion of the preceding paragraph.
\end{proof}

Now we return to:

\begin{proof}[Proof of \Cref{thm:main_rephrased}]
 Recall our assumption that  $\alpha \in\Gamma$ is an element such that $\rho_i(\alpha)$ is a rank one isometry of $X_i$ for $i=1,2$.
 We will show that for any infinite order element $\beta\in\Gamma$,
 $
     \slp (\rho(\beta)) = \slp (\rho(\alpha)).
 $

 In the $\rho(\G)$-invariant convex subset $C\subset X_1\times X_2$, we choose unit speed geodesic lines $l_{\rho(\alpha)}: \R\to C$ and $l_{\rho(\beta)}: \R\to C$, which are axes for $\rho(\alpha)$ and $\rho(\beta)$, respectively, and which are contained in unique flats of the form given by \eqref{eqn:axis_in_flat}:
 \begin{equation}
     l_{\rho(\alpha)} \subset F_{\rho(\alpha)} = l_{\rho_1(\alpha)} \times l_{\rho_2(\alpha)} 
     \quad\text{and}\quad
     l_{\rho(\beta)} \subset F_{\rho(\beta)} = l_{\rho_1(\beta)} \times l_{\rho_2(\beta)}.
 \end{equation}
 For convenience, we choose the origin of  $F_{\rho(\alpha)}$ as our base-point $x = (x_1,x_2)\in X_1\times X_2$; i.e., $x = l_{\rho(\alpha)}(0)$.

Let $\alpha_i^{\pm}, \beta_i^{\pm}\in \geo X_i$ denote the ideal endpoints of the geodesic axes $l_{\rho_i(\alpha)}$ and $l_{\rho_i(\beta)}$ for $i=1,2$. Based on the relative positions of these endpoints, we have the following two cases:
\begin{enumerate}
    \item For all $i\in\{1, 2\}$, we have $\{\alpha_i^{\pm}\}=\{\beta_i^{\pm}\}$. 
    \item For some $i\in\{1, 2\}$, we have $\{\alpha_i^{\pm}\}\ne\{\beta_i^{\pm}\}$.
    Without loss of generality, let us assume that $\beta_1^{+}\notin \{\alpha_1^{\pm}\}$. 
\end{enumerate}

In case (1), the conclusion $\slp(\rho(\alpha)) = \slp(\rho(\beta))$ follows from \Cref{lem:thm:main_rephrased}.

In case (2), we divide the proof into two subcases depending on the relative positions between $\beta_2^{+}$ and $\alpha_2^{-}$:
If $\beta_2^{+}$ is distinct from $\alpha_2^{-}$, then for both $i=1,2$, $\rho_i(\alpha^k)\cdot \beta_i^+ \to \alpha_i^+$ as $k\to\infty$. If $\beta_2^{+}$ is the same as $\alpha_2^{-}$, then $\rho_i(\alpha^{-k})\cdot \beta_i^+ \to \alpha_i^-$ as $k\to \infty$. In the former case, we will prove that $\slp(\rho(\alpha))=\slp(\rho(\beta))$.
In the latter case, we will prove $\slp(\rho(\alpha^{-1}))=\slp(\rho(\beta))$. Note that $\slp(\rho(\alpha))=\slp(\rho(\alpha^{-1}))$. So, in either case, we will have $\slp(\rho(\alpha))=\slp(\rho(\beta))$.

We only discuss the first subcase, i.e., we assume that $\beta_2^{+}\ne \alpha_2^{-}$.  The proof of the second subcase will be the same after replacing $\alpha$ by $\alpha^{-1}$. 

 Let $r_{k,i}: [0,\infty) \to X_i$ denote the unit speed geodesic ray  in $X_i$ emanating from the base-point $x_i$ and asymptotic to $\rho_i(\alpha^k)\cdot \beta_i^+$.
 Then the sequence of rays $(r_{k,i})$ converges (uniformly on compacts) to the ray $r_i \coloneqq l_{\rho_i(\alpha)}\vert_{[0,\infty)}$. 
 
 We claim that the sequence of maps
 \[
  r_{k,1}\times r_{k,2}: [0,\infty)^2 \to X_1\times X_2
 \]
 converges (uniformly on compacts) to $r_1 \times r_2$.
 Pick $T_1, T_2\in [0, \infty)$, and arbitrarily $\epsilon>0$. It suffices to prove that 
 $$d\bigl((r_{k, 1}(t_1), r_{k, 2}(t_2)), (r_1(t_1), r_2(t_2))\bigr)=\sqrt{d_1^2(r_{k, 1}(t_1), r_1(t_1))+d_2^2(r_{k, 2}(t_2), r_2(t_2))}<\epsilon$$
 for all $t_i\in [0,T_i]$ and all sufficiently large $k$. This follows from the fact that $r_{k, i}$ converges uniformly to $r_i$ for $i=1, 2$, and $d_i(r_{k, i}(t_i), r_i(t_i))<\epsilon/\sqrt 2$ for all $t_i\in [0,T_i]$ when $k$ is sufficiently large.

 For each $k$,  let $b_k :[0,\infty) \to C\subset X_1\times X_2$ be the unit speed ray emanating from $x$ and  asymptotic to $\rho(\alpha)^k\cdot l_{\rho(\beta)}\vert_{[0,\infty)}$.
 By \Cref{lemma:slp}, 
 $
     \slp(b_k) = \slp(\rho(\alpha)^k\cdot l_{\rho(\beta)}) = \slp(\l_{\rho(\beta)})\eqqcolon \sigma.
 $
 The (image of the) projection of $b_k$ to $X_i$ is the ray $r_{k,i}$.
 Thus, for all $k\in\N$,
 \[ 
 b_k(t) = \left (r_{k,1}\left (\frac{\sigma t}{\sqrt{1+\sigma^2}}\right),\, r_{k,2}\left (\frac{t}{\sqrt{1+\sigma^2}}\right)\right)
 \]
 By above, the sequence $b_k$ converges uniformly to the ray
 \[
  b(t) = \left (r_1\left (\frac{\sigma t}{\sqrt{1+\sigma^2}}\right),\,r_2\left (\frac{t}{\sqrt{1+\sigma^2}}\right)\right),
 \]
 whose slope is also $\sigma$.
 By the previous paragraph, the image of $b$ lies in $C \cap F_{\rho(\alpha)}$.
 Since $C\cap F_{\rho(\alpha)}$ contains an axis of $\rho(\alpha)$, \Cref{lem:thm:main} guarantees that $\sigma = \slp (\rho(\alpha))$.
\end{proof}

 \bibliography{bibliography.bib} 
\bibliographystyle{plain}

\vspace{.3in}

\noindent
{Max Planck Institute for Mathematics in the Sciences,
    Inselstra\ss e 22,
    04103 Leipzig, Germany
}

\noindent\texttt{subhadip.dey@mis.mpg.de}

\vspace{.2in}

\noindent
Department of Mathematics,
   The Ohio State University, 
   100 Math Tower, 231 W 18th Ave, Columbus, OH 43201, USA

\noindent\texttt{liu.11302@osu.edu}

\end{document}